\documentclass[11pt,letterpaper]{article}

\usepackage{fullpage}
\usepackage{amsmath,amsthm,amssymb}
\usepackage{color}

\newtheorem{theorem}{Theorem}
\newtheorem{proposition}[theorem]{Proposition}
\newtheorem{corollary}[theorem]{Corollary}

\numberwithin{equation}{section}

\allowdisplaybreaks[1]

\long\def\symbolfootnote[#1]#2{\begingroup%
\def\thefootnote{\fnsymbol{footnote}}\footnote[#1]{#2}\endgroup}

\newcommand{\E}{\mathbb{E}}
\newcommand{\Var}{{\rm Var}}
\newcommand{\RP}{\mathcal{S}}  
\newcommand{\NRP}{\mathcal{AS}} 
\newcommand{\LP}{\mathcal{A}}  
\newcommand{\SSP}{\mathcal{SS}} 

\newcommand{\floor}[1]{{\bigg\lfloor#1\bigg\rfloor}}
\newcommand{\ceiling}[1]{{\left\lceil#1\right\rceil}}
\newcommand{\stack}[2]{\genfrac{}{}{0pt}{1}{#1}{#2}}

\allowdisplaybreaks[3]

\begin{document}

\title{The mean and variance of the reciprocal merit factor of four classes of binary sequences}
\author{Jonathan Jedwab}
\date{16 March 2024}
\maketitle

\symbolfootnote[0]{
J.~Jedwab is with Department of Mathematics, 
Simon Fraser University, 8888 University Drive, Burnaby BC V5A 1S6, Canada.
He is supported by NSERC.
Email: {\tt jed@sfu.ca}

}

\begin{abstract}
The merit factor of a $\{-1, 1\}$ binary sequence measures the collective smallness of its non-trivial aperiodic autocorrelations. Binary sequences with large merit factor are important in digital communications because they allow the efficient separation of signals from noise. It is a longstanding open question whether the maximum merit factor is asymptotically unbounded and, if so, what is its limiting value. Attempts to answer this question over almost sixty years have identified certain classes of binary sequences as particularly important: skew-symmetric sequences, symmetric sequences, and anti-symmetric sequences.
Using only elementary methods, we find an exact formula for the mean and variance of the reciprocal merit factor of sequences in each of these classes, and in the class of all binary sequences. This provides a much deeper understanding of the distribution of the merit factor in these four classes than was previously available. A consequence is that, for each of the four classes, the merit factor of a sequence drawn uniformly at random from the class converges in probability to a constant as the sequence length increases.
\end{abstract}

\section{Introduction}
We consider the class of \emph{length $n$ binary sequences $\LP_n$}, namely 
$n$-tuples $(a_0, a_1, \dots, a_{n-1})$ (where $n>1$) having all entries $a_j$ in~$\{-1, 1\}$. 
The \emph{aperiodic autocorrelation} of the binary sequence $A$ at shift $u$ is given by
\[
C_A(u) := \sum_{j=0}^{n-u-1} a_j a_{j+u} \quad \mbox{for $0 \le u < n$},
\]
which measures the extent to which the sequence $A$ resembles itself when shifted by $u$ positions. 
Binary sequences whose aperiodic autocorrelations for $u \ne 0$ are collectively small 
have played a prominent role in digital communications engineering since the 1950s, because such sequences allow the efficient separation of signals from noise; see \cite{barker-alternatives} and \cite{schmidt-smallcor} for surveys.
One of the principal measures of this collective smallness is 
how small the \emph{peak sidelobe level}  $\max_{0 < u < n} |C_A(u)|$ can be for $A \in \LP_n$.
Our interest in this paper is the other principal measure of collective 
smallness, namely how large the
\emph{merit factor}
\begin{equation}
F(A) := \frac{n^2}{2\sum_{u=1}^{n-1}C_A(u)^2}
\label{Fdefn}
\end{equation}
(defined by Golay \cite{golay-defn} in 1972) can be for $A \in \LP_n$.
See \cite{merit-survey} for a survey of the merit factor and its importance in practical digital communications, and for an equivalent formulation in terms of the $L_4$ norm of complex-valued polynomials with $\pm 1$ coefficients on the unit circle.

Let $F_n = \max_{A \in \LP_n} F(A)$ be the maximum value of the merit factor over all binary sequences of length $n$. The overall goal in the study of the merit factor is to understand the asymptotic optimal behaviour by determining the value of $\limsup_{n \to \infty} F_n$. In 1966, Littlewood \cite[\S 6]{littlewood-poly} conjectured (using the $L_4$ norm formulation) that 
$\limsup_{n \to \infty} F_n$ is infinite, whereas in 1982 Golay~\cite{golay-long} (apparently without being aware of Littlewood's prior study) conjectured to the contrary that $\limsup_{n \to \infty} F_n$ is finite. 
This remains unresolved.

Several studies have identified an important subset of $\LP_n$:
the class of \emph{skew-symmetric} length~$n$ binary sequences 
\begin{equation}
\SSP_n := \left\{ (a_0,a_1,\dots,a_{n-1}) \in \LP_n: \mbox{ $n$ is odd and } 
a_j = (-1)^{j+\frac{n-1}{2}} a_{n-1-j} \quad \mbox{for $0 \le j <n$}
\right\}.
\label{SSPdefn}
\end{equation}
Indeed, Golay~\cite{golay-long} conjectured that 
\[
\limsup_{n\to\infty}\max_{A\in\SSP_n} F(A) = \limsup_{n\to\infty}\max_{A\in\LP_n} F(A),
\]
which would imply that to determine the asymptotic largest merit factor it is sufficient to restrict attention from $\LP_n$ to the class~$\SSP_n$. 
For twenty-five years from 1988, the best known asymptotic result was the construction of an infinite family of binary sequences having asymptotic merit factor~$6$; the same value of 6 is also attained asymptotically by certain families of skew-symmetric sequences of length $2p+1$ or $4p+1$, where $p$ is a prime congruent to 1 modulo~4 \cite[Corollaries 6 and 9]{schmidt-jedwab-parker-twofam}. Then, in 2013, the asymptotic value $6$ was improved to the value $\frac{1}{c-1} > 6.34$, where $c$ is the smallest root of $27x^3-498x^2+1164x-722$~\cite{jedwab-katz-schmidt-6.34}. It was subsequently shown that this larger value can also be attained asymptotically by skew-symmetric sequences~\cite{jedwab-katz-schmidt-advances}.

This provides clear motivation to understand the distribution of the merit factor as a function of length $n$, in both the classes $\LP_n$ and~$\SSP_n$.
In this paper, we shall show how to determine the exact mean and variance of 
$\frac{1}{F(A)}$ for a sequence $A$ drawn uniformly at random from 
$\LP_n$ and~$\SSP_n$, as set out in Theorems~\ref{thm:all-mean-var} and~\ref{thm:skew-symm-mean-var}.

\begin{theorem}
\label{thm:all-mean-var}
Let $A$ be drawn uniformly at random from~$\LP_n$. Then
\begin{align*}
n^2\, \E \Big ( \frac{1}{F(A)} \Big ) &= n^2 -n, \\
n^4\, \Var \Big ( \frac{1}{F(A)} \Big ) &=
   \displaystyle{\frac{16}{3}n^3 - 20n^2 + \frac{56}{3}n 
   - 2 + 2(-1)^n}.
\end{align*}
\end{theorem}

\begin{theorem}
\label{thm:skew-symm-mean-var}
Let $n$ be odd and let $A$ be drawn uniformly at random from~$\SSP_n$. Then
\begin{align*}
n^2\, \E \Big ( \frac{1}{F(A)} \Big ) &= n^2 -3n + 2, \\
n^4\, \Var \Big ( \frac{1}{F(A)} \Big ) &= 
\frac{32}{3}n^3 - 88n^2 + \frac{592}{3}n 
-512\floor{\frac{n-1}{8}}
-512\floor{\frac{n-1}{12}}
-88+16(-1)^\frac{n-1}{2}(n-3).
\end{align*}
\end{theorem}

Two other subsets of $\LP_n$ arise naturally in the study of the merit factor: the class of \emph{symmetric} length $n$ binary sequences
\begin{equation}
\RP_n := \left\{ (a_0,a_1,\dots,a_{n-1}) \in \LP_n: 
a_j = a_{n-1-j} \quad \mbox{for $0 \le j <n$} \right\},
\label{RPdefn}
\end{equation}
and the class of \emph{anti-symmetric} length $n$ binary sequences
\begin{equation}
\NRP_n := \left\{ (a_0,a_1,\dots,a_{n-1}) \in \LP_n: 
\mbox{ $n$ is even and } 
a_j = -a_{n-1-j} \quad \mbox{for $0 \le j <n$} \right\}.
\label{NRPdefn}
\end{equation}
For example, the asymptotic value~$6$ for $F(A)$ is attained not only by skew-symmetric sequences but also by binary sequences formed by prepending the element $1$ to a sequence $g$ of length $n-1$ and then cyclically rotating by  $\lfloor n/4 \rfloor$ positions, for certain symmetric $g$ when $n$ is a prime congruent to 1 modulo~4 and for certain anti-symmetric $g$ when $n$ is a prime congruent to 3 modulo~4~\cite{hoholdt-jensen}. 
Furthermore, every skew-symmetric binary sequence can be written as the interleaving of sequences $f$ and $g$, where one of $f,g$ is symmetric and the other is anti-symmetric. By choosing such sequences $f$ and~$g$ each to have large merit factor, Golay and Harris~\cite{golay-harris} constructed examples of sequences in $\SSP_n$ with large merit factor for each odd~$n$ in the range $71 \le n \le 117$; these values of the merit factor were later shown to attain the actual maximum over $\SSP_n$ for all but two values of $n$ in this range~\cite{packebusch-mertens}.
In contrast to the class $\LP_n$, the merit factor of sequences in $\RP_n$ is known to be bounded above: Fredman, Saffari and Smith \cite{fredman-saffari-smith} proved that $F(A) \le 0.1048^{-1} < 10$ for all $A \in \RP_n$.

We shall provide a counterpart to Theorems~\ref{thm:all-mean-var} and~\ref{thm:skew-symm-mean-var} by determining the exact mean and variance of $\frac{1}{F(A)}$ for a sequence $A$ drawn uniformly at random from either of the classes $\RP_n$ and $\NRP_n$, as set out in Theorem~\ref{thm:symm-mean-var} (in which $I[\,\cdot\,]$ denotes the indicator function of an event).
\begin{theorem}
\label{thm:symm-mean-var}
Let $f$ be drawn uniformly at random from $\RP_n$, or (for even~$n$) from~$\NRP_n$. Then 
\begin{align*}
n^2\, \E \Big ( \frac{1}{F(A)} \Big ) &= 2n^2 -3n+\frac{1-(-1)^n}{2}, \\
n^4\, \Var \Big ( \frac{1}{F(A)} \Big ) &= 
\begin{cases}
   \displaystyle{32n^3 - 216n^2 + 304n 
   +256\floor{\frac{n}{6}}
   +256 \cdot I[n \bmod 6 = 4]} 	& \mbox{for $n$ even,} \\[2ex]
   \displaystyle{32n^3 - 144n^2 + 160n 
   -576\floor{\frac{n-1}{4}}
   -512\floor{\frac{n-1}{6}}}
   -48					& \mbox{for $n$ odd}.
\end{cases}
\end{align*}
\end{theorem}

The expected value of $\frac{1}{F(A)}$, as given in Theorems~\ref{thm:all-mean-var}, \ref{thm:skew-symm-mean-var}, and~\ref{thm:symm-mean-var}, was previously calculated for $\LP_n$ \cite{sarwate} (often attributed instead to \cite{newman-byrnes}), for $\SSP_n$ \cite{borwein-book}, and for $\RP_n$ and $\NRP_n$ \cite{borwein-choi-average-norm}. An expression for $\Var(\frac{1}{F(A)})$ for the class $\LP_n$ was stated in \cite{aupetit-liardet-slimane}, but was incorrectly calculated (see the discussion following the proof of Theorem~\ref{thm:all-mean-var}). 
Katz and Ramirez \cite{katz-ramirez}, with reference to a preprint version \cite{jedwab-L4norm} of the current paper, recently calculated $\E \big(\frac{1}{F(A)} - \E(\frac{1}{F(A)}) \big) ^p$ for $p = 2,3,4$, where $A$ is drawn uniformly at random from $\LP_n$, and in the case $p=2$ obtained the same variance expression as in Theorem~\ref{thm:all-mean-var}.
None of the variance expressions for the classes $\SSP_n$, $\RP_n$, and $\NRP_n$ was previously known, and their derivation is considerably more difficult than that for~$\LP_n$. We find it surprising and remarkable that these exact closed form expressions exist, and that they can be determined precisely using only elementary methods.
These expressions provide a much deeper understanding of the distribution of the merit factor in the four classes $\LP_n$, $\SSP_n$, $\RP_n$, $\NRP_n$ than was previously available.
The expressions in Theorem~\ref{thm:all-mean-var} have been verified numerically for $n \le 40$, and those in Theorems~\ref{thm:skew-symm-mean-var} and~\ref{thm:symm-mean-var} for $n \le 75$.

Application of Chebyshev's inequality to Theorems~\ref{thm:all-mean-var}, \ref{thm:skew-symm-mean-var} and~\ref{thm:symm-mean-var} gives the following corollary.
\begin{corollary}
\label{cor:conv-binseq}
Let $A$ be drawn uniformly at random from $\LP_n$ or (for odd~$n$) from $\SSP_n$, and let $B$ be drawn uniformly at random from $\RP_n$ or (for even~$n$) from~$\NRP_n$. Then, as $n \to \infty$,
\[
\mbox{$F(A) \to 1$ in probability, and $F(B) \to \frac{1}{2}$ in probability.} 
\]
\end{corollary}
\noindent
Of the asymptotic results given in Corollary~\ref{cor:conv-binseq}, only that for $\LP_n$ was previously known \cite{borwein-lockhart-expected-norm}; our techniques are very different.
Corollary~\ref{cor:conv-binseq} highlights the difficulty of trying to determine which sequences of length $n$ attain the largest merit factor: they have zero density within their respective class as $n \to \infty$.
Indeed, a value of $\limsup_{n \to \infty} F(A)$ larger than $6.34$ can be attained for $A$ in $\LP_n$ and $\SSP_n$ \cite{jedwab-katz-schmidt-6.34,jedwab-katz-schmidt-advances} (as already noted), and 
a value of $\limsup_{n \to \infty} F(B)$ equal to $1.5$ can be attained for $B$ in $\RP_n$ and~$\NRP_n$ (by taking the offset fraction in the main result of \cite{hoholdt-jensen} to be~$0$).

In Section~\ref{sec:summation} we give some preliminary results for use in later calculations. In Sections~\ref{sec:all}, \ref{sec:symmetric}, and~\ref{sec:skew-symmetric} we prove Theorems~\ref{thm:all-mean-var}, ~\ref{thm:symm-mean-var}, and~\ref{thm:skew-symm-mean-var}, respectively.

\section{Summation and calculation identities}
\label{sec:summation}
Throughout the paper, we denote the indicator function of the event~$X$ by $I[X]$, and write $I_u:= I[\mbox{$u$ odd}]$. We shall make use of the following summation identities, of which \eqref{sumflooru} to \eqref{sumIuu-1u-3} can readily be verified by considering the cases $n$ even and $n$ odd separately; \eqref{sumIuIu+2v} follows from \eqref{sumIuu-1}; \eqref{sumIu2u-n-1} can be verified by substituting $u=2U+1$ and considering the cases $n \bmod 4 = 1$ and $n \bmod 4 = 3$ separately; \eqref{sumIuIvI2u+v} follows from \eqref{sumIu2u-n-1} by substituting $v=2V+1$; \eqref{sum3n+14} can be verified by considering the cases $n \bmod 8 = 1,3,5,7$ separately; and \eqref{sum2n+13} can be verified by considering each of the congruence classes of $n$ modulo 6 separately:
\begin{align}
\sum_{u=1}^{n-1} \floor{\frac{u}{2}} 
  &= \floor{\frac{n}{2}} \floor{\frac{n-1}{2}},		\label{sumflooru} \\
3\sum_{u=1}^{n-1} \floor{\frac{u}{2}}\floor{\frac{u-2}{2}}
  &= 2\floor{\frac{n}{2}} \left(\frac{n-2}{2}\right) \ceiling{\frac{n-4}{2}},
			 				\label{sumflooruu-2} \\
\sum_{u=1}^{n-1} I_u 
  & = \floor{\frac{n}{2}},				\label{sumIu} \\
2\sum_{u=1}^{n-1}I_u \left ( \frac{u-1}{2} \right ) 
  &= \floor{\frac{n}{2}} \floor{\frac{n-2}{2}}, 	\label{sumIuu-1} \\
3\sum_{u=1}^{n-1}I_u \left(\frac{u-1}{2}\right) \left(\frac{u-3}{2}\right)
  &= \floor{\frac{n}{2}} \floor{\frac{n-2}{2}} \floor{\frac{n-4}{2}},
							\label{sumIuu-1u-3} \\
2\sum_{u,v=1}^{n-1}I_uI[u+2v>2n]
  &= \floor{\frac{n}{2}}\floor{\frac{n-2}{2}}, 		\label{sumIuIu+2v} \\
\sum_{u=\frac{n+1}{2}}^{n-1} I_u\left ( \frac{2u-n-1}{2} \right ) 
  &= \floor{\frac{n-1}{4}} \floor{\frac{n-3}{4}} \quad \mbox{for $n$ odd}, 
							\label{sumIu2u-n-1} \\
\sum_{u,v=1}^{n-1} I_uI_vI[2u+v>2n]
  &= \floor{\frac{n-1}{4}} \floor{\frac{n-3}{4}} \quad \mbox{for $n$ odd}, 
							\label{sumIuIvI2u+v} \\
\sum_{u=\ceiling{\frac{3n+1}{4}}}^{n-1}I_u 
  &= \floor{\frac{n-1}{8}} 			 \quad \mbox{for $n$ odd},
					 		\label{sum3n+14} \\
\sum_{u=\ceiling{\frac{2n+1}{3}}}^{n-1}I_u 
  &= \floor{\frac{n}{6}}+I[n \bmod 6 = 4]. 		\label{sum2n+13}
\end{align}

We next obtain expressions for $n^2\, \E (\frac{1}{F(A)})$ and $n^4\, \Var (\frac{1}{F(A)})$ for $A = (a_0,a_1,\dots,a_{n-1}) \in \LP_n$. We write
\begin{equation}
C_u := \sum_{j=0}^{u-1} a_j a_{j+n-u} \quad \mbox{for $0 < u < n$},
\label{Cdefn}
\end{equation}
regarding the sequence entries $a_j$ as random variables each taking values in $\{1,-1\}$.

\begin{proposition}
\label{prop:expressions}
Let $A = (a_0,a_1,\dots,a_{n-1}) \in \LP_n$ be drawn at random according to an arbitrary distribution on the sequence entries $a_j$. Then
\begin{align}
n^2\, \E \Big(\frac{1}{F(A)}\Big) &= 2E, \label{Ef44} \\
n^4\, \Var \Big(\frac{1}{F(A)}\Big) &= 4(V-E^2), \label{Vf44}
\end{align}
where
\begin{align}
E &= \sum_{u=1}^{n-1} \E \, C_u^2,  				\label{E} \\
V &= \sum_{u,v=1}^{n-1} \E \big ( C_u^2 C_v^2 \big ).		\label{V} 
\end{align}
\end{proposition}

\begin{proof}
Since $C_u = C_A(n-u)$, by the definition \eqref{Fdefn} of $F(A)$ we have
\begin{equation}
\frac{n^2}{F(A)} = 2\sum_{u=1}^{n-1} C_u^2.
\label{normCA}
\end{equation}
Take the expectation to obtain \eqref{Ef44}.
Take the variance and substitute from \eqref{E} to obtain~\eqref{Vf44}.
\end{proof}

\section{The class $\LP_n$}
\label{sec:all}
In this section, we use Proposition~\ref{prop:expressions} to prove Theorem~\ref{thm:all-mean-var} for the class~$\LP_{n}$ of binary sequences.

\begin{proof}[Proof of Theorem~$\ref{thm:all-mean-var}$]
Let $A = (a_0,a_1,\dots,a_{n-1}) \in \LP_n$, and regard the sequence entries $a_j$ as independent random variables that each take the values $1$ and $-1$ with probability~$\frac{1}{2}$. Substitute the definition \eqref{Cdefn} of $C_u$ into expression \eqref{E} for $E$ to give
\[
E = \sum_{u=1}^{n-1} \, \sum_{j,k=0}^{u-1} \E(a_j a_{j+n-u} a_k a_{k+n-u}).
\]
Since the $a_i$ are independent, the expectation term in the triple sum is nonzero exactly when $\{j, j+n-u, k, k+n-u\} = \{i, i, i', i'\}$ for some indices $i, i'$, or equivalently when $j=k$. Therefore
$E = \sum_{u=1}^{n-1} \sum_{j=0}^{u-1} \E(1) = n(n-1)/2$, and then \eqref{Ef44} gives 
\[
n^2\, \E\Big(\frac{1}{F(A)}\Big) = n^2-n, 
\]
as required.

Substitute the definition \eqref{Cdefn} of $C_u$ into expression \eqref{V} for $V$ to give
\[
V = 
 \sum_{u,v=1}^{n-1} \, \sum_{j,k=0}^{u-1} \, \sum_{\ell,m=0}^{v-1}
 \E \big(
 a_j a_{j+n-u} a_k a_{k+n-u} a_\ell a_{\ell+n-v} a_m a_{m+n-v}
 \big ).
\]
There are four mutually disjoint cases for which the expectation term in the sum $V$ is nonzero.
\begin{description}

\item[Case 1:] $u=v$ and $j=k=\ell=m$.
The contribution to $V$ is 
$\sum_{u=1}^{n-1} \sum_{j=0}^{u-1} 1 = \sum_{u=1}^{n-1} u$.

\item[Case 2:] $u=v$ and $\{j,k,\ell,m\} = \{i,i,i',i'\}$ for some indices $i \neq i'$.
The contribution to $V$ (from 3 symmetrical cases) is
\[
  3 \sum_{u=1}^{n-1} \sum_{j=0}^{u-1} \sum_{\stack{k=0}{k \ne j}}^{u-1} 1
= 3 \sum_{u=1}^{n-1} u(u-1).
\]

\item[Case 3:] $u\ne v$ and $j=k$ and $\ell = m$.
The contribution to $V$ is
\[
  \sum_{u=1}^{n-1} \sum_{\stack{v =1}{v \ne u}}^{n-1}
   \sum_{j=0}^{u-1} \sum_{\ell=0}^{v-1} 1
= \sum_{u=1}^{n-1} \sum_{\stack{v =1}{v \ne u}}^{n-1} u v
= \bigg(\sum_{u=1}^{n-1} u\bigg)^2 - \sum_{u=1}^{n-1} u^2.
\]

\item[Case 4:] $u\ne v$ and $\{j,k\} = \{i,i+n-v\}$ and $\{\ell,m\} = \{i,i+n-u\}$ for some index $i$ satisfying $0 \le i < u+v-n$.
The contribution to $V$ (from 4 symmetrical cases) is
\[
  4 \sum_{u=1}^{n-1} \sum_{\stack{v=1}{v \ne u}}^{n-1} \sum_{j=0}^{u+v-n-1} 1
= 4 \sum_{u=1}^{n-1} \sum_{\stack{v=n-u+1}{v \ne u}}^{n-1} (u+v-n)
= 4 \sum_{u=1}^{n-1} \sum_{\stack{w=1}{w \ne 2u-n}}^{u-1} w,
\]
putting $w = u+v-n$. The condition $w \ne 2u-n$ in the inner sum takes effect only when $2u-n \ge 1$, so the contribution for this case equals
\[
  4 \sum_{u=1}^{n-1} \sum_{w=1}^{u-1} w 
  - 4 \sum_{u=\ceiling{\frac{n+1}{2}}}^{n-1} (2u-n)
= 2 \sum_{u=1}^{n-1} u(u-1)
  - (n^2-2n+ I_n)
\]
by evaluating the sum involving $2u-n$ separately according to whether $n$ is even or odd.
\end{description}

Sum the contributions to $V$ from the four cases and substitute into \eqref{Vf44}, together with the relation $E = n(n-1)/2$ already calculated, to give
\begin{align*}
n^4\, \Var\Big(\frac{1}{F(A)}\Big) 
 &= 16\sum_{u=1}^{n-1} u^2 - 16\sum_{u=1}^{n-1}u 
  + 4\bigg(\sum_{u=1}^{n-1}u\bigg)^2
  - 4(n^2-2n+ I_n) -n^2(n-1)^2 \\
 &= \frac{16}{3}n^3 - 20n^2 + \frac{56}{3}n 
   -2 + 2(-1)^n,
\end{align*}
as required.
\end{proof}

The calculations in the proof of Theorem~\ref{thm:all-mean-var} follow the general method of Aupetit, Liardet and Slimane \cite[\S 2]{aupetit-liardet-slimane}, but correct the mistaken conclusion of \cite[p.44, l.6]{aupetit-liardet-slimane} that (in our notation) $n^4\, \Var(\frac{1}{F(A)}) = \frac{8}{3}n(n-1)(n+4)$. The mistake arises from failing to apply the condition $r+s < N$ in computing the summation of \cite[p.44, l.2]{aupetit-liardet-slimane}, which corresponds in our notation to neglecting the condition $u+v > n$ in the analysis of Case~4.

\section{The classes $\RP_{n}$ and $\NRP_{n}$}
\label{sec:symmetric}
In this section, we use Proposition~\ref{prop:expressions} to prove 
Theorem~\ref{thm:symm-mean-var} for the class $\RP_n$ of symmetric binary sequences and for the class $\NRP_n$ of anti-symmetric binary sequences.

\begin{proof}[Proof of Theorem~$\ref{thm:symm-mean-var}$]
By \eqref{RPdefn} and \eqref{NRPdefn}, the map sending the sequence $A = (a_j)$ to the sequence $B = ((-1)^ja_j)$ is an involution between $\RP_{2m}$ and $\NRP_{2m}$ satisfying $F(A) = F(B)$.
It is therefore sufficient to consider the class $\RP_n$.

Let $(a_0,a_1,\dots,a_{n-1}) \in \RP_n$. 
By the definition~\eqref{RPdefn} of $\RP_n$, the $a_j$ satisfy the symmetry condition
\begin{equation}
a_j = a_{n-1-j} \quad \mbox{for $0 \le j <n$}.
\label{symm-condition}
\end{equation}
We regard the sequence entries $a_0, a_1, \dots, a_{\lfloor \frac{n-1}{2} \rfloor}$ as independent random variables that each take the values $1$ and $-1$ with probability $\frac{1}{2}$, and the remaining sequence entries as being determined by~\eqref{symm-condition}. 

Set
\begin{equation}
D_u:= 2\!\!\sum_{0 \le j < \frac{u-1}{2}} a_j a_{j+n-u},
\label{Ddefn}
\end{equation}
and use condition \eqref{symm-condition} to rewrite \eqref{Cdefn} as
\[
C_u = I_u + D_u, 
\]
where the term $I_u$ arises from the product $a_{(u-1)/2}^2$ when $u$ is odd.
Substitute for $C_u$ in \eqref{E} and expand to give
\begin{equation}
E = \floor{\frac{n}{2}} + 2\sum_{u=1}^{n-1} I_u \, \E D_u + \sum_{u=1}^{n-1} \E  D_u^2, 				\label{E-symm} 
\end{equation}
using the summation identity~\eqref{sumIu}. 
Similarly substitute for $C_u$ and $C_v$ in \eqref{V} and expand; using symmetry in $u$ and $v$, and the summation identity \eqref{sumIu}, we obtain
\begin{align}
V &= \floor{\frac{n}{2}}^2
   +4\floor{\frac{n}{2}} \sum_{u=1}^{n-1} I_u \, \E D_u
   +4 \sum_{u,v=1}^{n-1} I_u I_v \, \E\big( D_uD_v \big) 
   +2\floor{\frac{n}{2}} \sum_{u=1}^{n-1} \E D_u^2 		\nonumber\\
  &\phantom{=} +4 \sum_{u,v=1}^{n-1} I_u \, \E\big( D_u D_v^2 \big) 
  +\sum_{u,v=1}^{n-1} \E\big( D_u^2 D_v^2 \big) 		\label{V-symm}.
\end{align}

We shall use \eqref{Ddefn} to express each of 
$\E D_u$,
$\E\big( D_u D_v \big )$,
$\E D_u^2$,
$\E\big( D_u D_v^2 \big )$, and
$\E\big( D_u^2 D_v^2 \big )$
as a sum of expectation terms of the form $\E(a_{j_1} a_{j_2} \dots a_{j_{2r}})$, where $1 \le r \le 4$. In view of the symmetry condition \eqref{symm-condition}, such a term is nonzero exactly when the indices $j_1, j_2, \dots, j_{2r}$ admit a \emph{matching decomposition}, namely a partition into $r$ pairs $\{j, k\}$ such that those pairs which are not \emph{equal} ($j=k$) are \emph{symmetric} ($j+k=n-1$, written as $j \sim k$). We shall identify the index sets admitting a matching decomposition, multiply the resulting expressions by $I_u$ or $I_v$ and sum over $u$ or $v$ in the range $1 \leq u, v \le n-1$ as appropriate, and then substitute into the forms \eqref{E-symm} and \eqref{V-symm} for $E$ and $V$ to calculate $n^2\, \E (\frac{1}{F(A)})$ and $n^4\, \Var (\frac{1}{F(A)})$ from Proposition~\ref{prop:expressions}.

We shall use the observations that, for $1 \le u, v \le n-1$,
\begin{align}
& \mbox{a pair (equal or symmetric) cannot be formed from indices $j, j+n-u$ satisfying $0 \le j < \frac{u-1}{2}$}, \label{one-pair}
\end{align}
\begin{align}
& \mbox{$j \not \sim k$ for indices $j, k$ satisfying $0 \le j < \frac{u-1}{2}$ and $0 \le k < \frac{v-1}{2}$} \label{symm-pair},
\end{align}
and
\begin{align}
& \mbox{the only pairs that can be formed from indices $j, j+n-u, k, k+n-u$ satisfying}		\nonumber \\
& \mbox{$0 \le j < k < \frac{u-1}{2}$ are $k=j+n-u$ (equal) and $j+n-u \sim k+n-u$ (symmetric).} \label{two-pairs}
\end{align}

\begin{description}

\item[The sum $\sum_{u=1}^{n-1}I_u\,\E D_u$.]
From \eqref{Ddefn} and \eqref{one-pair}, this sum is
\begin{equation}
\sum_{u=1}^{n-1}I_u\,\E D_u =
2\sum_{u=1}^{n-1}\,I_u\!\!\sum_{0 \le j < \frac{u-1}{2}} \E(a_j a_{j+n-u}) = 0.
\label{sum1}
\end{equation}

\item[The sum $\sum_{u,v=1}^{n-1}I_uI_v\,\E\big( D_u D_v \big)$.]
From \eqref{Ddefn}, this sum equals
\[
4 \sum_{u,v=1}^{n-1} \,I_uI_v\!\!\sum_{0 \le j < \frac{u-1}{2}} \, \sum_{0 \le k < \frac{v-1}{2}} \E (a_j a_{j+n-u} a_k a_{k+n-v}).
\]
The indices $j, j+n-u, k, k+n-v$ admit a matching decomposition in two ways, represented in Cases 1 and 2 below. All other possible arrangements of these four indices into two pairs are inconsistent with the given ranges for $u,v,j,k$, either because of a single pairing excluded by \eqref{one-pair} or \eqref{symm-pair}, or else by combination of two pairings as summarized in Table~\ref{tab:inconsistent11}.

\begin{description}

\item[Case 1:] $j=k$ and $j+n-u = k+n-v$. This gives $j=k$ and $u=v$, and the contribution to the sum is
\[
4 \sum_{u=1}^{n-1}\,I_u\!\! \sum_{0 \le j < \frac{u-1}{2}} 1
= 4 \sum_{u=1}^{n-1} I_u \left ( \frac{u-1}{2} \right ) 
= 2 \floor{\frac{n}{2}} \floor{\frac{n-2}{2}}
\]
using the summation identity~\eqref{sumIuu-1}.

\item[Case 2:] $j=k$ and $j+n-u \sim k+n-v$. This gives $j=k = \frac{u+v-n-1}{2}$. These values of $j, k$ have already been counted as part of Case 1 when $u=v$, so we impose the constraint $u \ne v$ and calculate the additional contribution to the sum from Case 2 as
\[
4\, I_n \sum_{\stack{u,v=1}{u \ne v}}^{n-1} I_uI_v \,\, I[u+v > n].
\]
Substitute $u = 2U+1$ and $2V+1$ to evaluate this additional contribution as
\begin{align*}
  4\, I_n \sum_{\stack{U,V=0}{U \ne V}}^{\frac{n-3}{2}} I[U+V > \tfrac{n-2}{2}]
& = 4\, I_n \sum_{U,V=0}^{\frac{n-3}{2}} I[U+V > \tfrac{n-2}{2}] - 4\, I_n \sum_{U=0}^{\frac{n-3}{2}} I[U > \tfrac{n-2}{4}] \\
& = 4\, I_n \frac{(n-1)(n-3)}{8} - 4\,I_n \floor{\frac{n-1}{4}}.
\end{align*}

\end{description}

\begin{table}[htb]
\centering
\caption{Inconsistent index pairings for the sum $\sum_{u,v=1}^{n-1}I_uI_v\,\E\big( D_u D_v \big)$}
\label{tab:inconsistent11}
\vspace{1ex}
\begin{tabular}{|ll|l|} 					   \hline
Index pairings 		& 			
	& Inconsistency from   					\\ \hline
$j=k+n-v$, 		& $k=j+n-u$ 		
	& $u+v = 2n$ 		\\[1ex]
$j=k+n-v$, 		& $k \sim j+n-u$ 	
	& $j = \frac{u-1}{2}+\frac{n-v}{2} > \frac{u-1}{2}$ 	\\[1ex]
$j \sim k+n-v$, 	& $k = j+n-u$ 	
	& $k = \frac{v-1}{2}+\frac{n-u}{2} > \frac{v-1}{2}$ 	\\[1ex]
$j \sim k+n-v$, 	& $k \sim j+n-u$ 	
	& $j+k = \frac{u-1}{2} + \frac{v-1}{2}$			\\[1ex] \hline
\end{tabular}
\end{table}

Combine Cases 1 and 2 to give
\begin{equation}
\sum_{u,v=1}^{n-1}I_uI_v\,\E\big( D_u D_v \big)
= 2 \floor{\frac{n}{2}} \floor{\frac{n-2}{2}} +
2\, I_n \left(\frac{n-1}{2}\right)\left(\frac{n-3}{2}\right) - 4\,I_n \floor{\frac{n-1}{4}}.
\label{sum11}
\end{equation}

\item[The sum $\sum_{u=1}^{n-1}\E D_u^2$.]
This sum arises by restricting to $u=v$ in the analysis of the previous sum (with the $I_u I_v$ term absent), and so is
\begin{equation}
\sum_{u=1}^{n-1}\E D_u^2 = 
4 \sum_{u=1}^{n-1} \,\, \sum_{0 \le j < \frac{u-1}{2}} 1 =
4 \sum_{u=1}^{n-1} \floor{\frac{u}{2}} = 
4 \floor{\frac{n}{2}}\floor{\frac{n-1}{2}}
\label{sum2}
\end{equation}
using the summation identity~\eqref{sumflooru}.

\item[The sum $\sum_{u,v=1}^{n-1} I_u \, \E\big( D_u D_v^2 \big)$.]
From \eqref{Ddefn}, this sum equals
\[
 8 \sum_{u,v=1}^{n-1}\, I_u\!\!\sum_{0 \le j < \frac{u-1}{2}} \, \sum_{0 \le k, \ell < \frac{v-1}{2}} \E (a_j a_{j+n-u} a_k a_{k+n-v} a_\ell a_{\ell+n-v}).
\]
The contributions from indices $k = \ell$ involve terms in $\E (a_j a_{j+n-u})$, which by~\eqref{one-pair} is zero. By symmetry in $k$ and $\ell$, it is therefore sufficient to take twice the contributions from indices $k < \ell$, namely
\[
 16 \sum_{u,v=1}^{n-1}\, I_u\!\!\sum_{0 \le j < \frac{u-1}{2}} \, \sum_{0 \le k < \ell < \frac{v-1}{2}} \E (a_j a_{j+n-u} a_k a_{k+n-v} a_\ell a_{\ell+n-v}).
\]
In order for the indices $j, {j+n-u}, k, {k+n-v}, \ell, {\ell+n-v}$ to admit a matching decomposition, at least one equal pair or one symmetric pair must be formed from the four indices $k, k+n-v, \ell, \ell+n-v$. In the given range $k < \ell < \frac{v-1}{2}$, from observation \eqref{two-pairs} either $\ell = k+n-v$ (represented in Cases 2 and 3 below) or $k+n-v \sim \ell+n-v$ (represented in Cases 1 and 4 below). We can now determine all matching decompositions by starting with each of these two pairs in turn, and identifying all possible arrangements of the four remaining unpaired indices into two further pairs consistent with the given ranges for $u,v,j,k,\ell$.
These arrangements are listed in Cases 1 to 4 below; all other arrangements are inconsistent, either because of a single pairing excluded by \eqref{one-pair} or \eqref{symm-pair}, or else by combination of two or more pairings as summarized in Table~\ref{tab:inconsistent12}.

\begin{description}

\item[Case 1:] $k+n-v \sim \ell+n-v$ and $j=k$ and $j+n-u \sim \ell$. This gives $2v-u=n$ and $j = k$ and $\ell = 2v-n-1-j$, and the contribution to the sum is
\[
 16\, I_n \sum_{v=\frac{n+1}{2}}^{n-1} \,\,\sum_{\stack{0 \le j < \frac{2v-n-1}{2}}{j > \frac{3v-2n-1}{2}}} 1.
\]

\item[Case 2:] $\ell=k+n-v$ and $j=k$ and $j+n-u = \ell+n-v$. This gives $2v-u=n$ and $j = k$ and $\ell = j+n-v$, and the contribution to the sum is
\[
 16\, I_n \sum_{v=\frac{n+1}{2}}^{n-1}\,\, \sum_{0 \le j < \frac{3v-2n-1}{2}} 1.
\]

\item[Case 3:] $\ell=k+n-v$ and $j=k$ and $j+n-u \sim \ell+n-v$.
This gives $j = k = \frac{u+2v-2n-1}{2}$ and $\ell=\frac{u-1}{2}$. These values of $j, k, \ell$ have already been counted as part of Case 2 when $2v-u=n$, so we impose the constraint $2v-u \ne n$ and evaluate the additional contribution to the sum from Case 3 as
\[
   16 \!\!\!\! \sum_{\stack{u,v=1}{2v-u \ne n}}^{n-1} \!\!\! I_u\,\, I[u+2v>2n]\,\, I[u<v].
\]

\item[Case 4:] $k+n-v \sim \ell+n-v$ and  $j=k$ and $j+n-u = \ell$. This gives $j = k = \frac{u+2v-2n-1}{2}$ and $\ell=\frac{2v-u-1}{2}$, and the contribution to the sum is 
\[
 16 \sum_{u,v=1}^{n-1}\, I_u\,\, I[u+2v>2n]\,\, I[u>v].
\]

\end{description}

\begin{table}[htb]
\centering
\caption{Inconsistent index pairings for the sum $\sum_{u,v=1}^{n-1} I_u \, \E\big( D_u D_v^2 \big)$}
\label{tab:inconsistent12}
\vspace{1ex}
\begin{tabular}{|lll|l|}						  \hline
Index pairings 		& 			&			
	& Inconsistency from 					       	\\\hline
$\ell=k+n-v$, 		& $j=\ell+n-v$, 	& $k=j+n-u$ 	
	& $u+2v = 3n$ 							\\[1ex]
$\ell=k+n-v$, 		& $j=\ell+n-v$, 	& $k \sim j+n-u$
	& $j = \frac{u-1}{2}+n-v > \frac{u-1}{2}$ 			\\[1ex]
$\ell=k+n-v$, 		& $j \sim \ell+n-v$, 	& $k=j+n-u$ 
	& $\ell = \frac{v-1}{2} + \frac{2n-u-v}{2} > \frac{v-1}{2}$ 	\\[1ex]
$\ell=k+n-v$, 		& $j \sim \ell+n-v$, 	& $k \sim j+n-u$
	& $j+\ell = \frac{u+n-2}{2} > \frac{u-1}{2} + \frac{v-1}{2}$ 	\\[1ex]
$k+n-v \sim \ell+n-v$, 	& $j=\ell$, 		& $k=j+n-u$ 
	& $k = \ell +n-u > \ell$ 					\\[1ex]
$k+n-v \sim \ell+n-v$, 	& $j=\ell$, 		& $k \sim j+n-u$
	& $k+\ell = u-1$ 						\\[1ex]
\hline
\end{tabular}
\end{table}

The combined contribution to the sum from Cases 1 and 2 is 
\[
 16\, I_n \sum_{v=\frac{n+1}{2}}^{n-1} \,\,\sum_{\stack{0 \le j < \frac{2v-n-1}{2}}{j \ne \frac{3v-2n-1}{2}}} 1,
\]
in which the condition $j \ne \frac{3v-2n-1}{2}$ in the inner sum takes effect only when $\frac{3v-2n-1}{2}$ is a non-negative integer. The combined contribution from Cases 1 and 2 therefore equals
\[
 16\, I_n \sum_{v=\frac{n+1}{2}}^{n-1} \frac{2v-n-1}{2} 
-16\, I_n\!\! \sum_{v=\ceiling{\frac{2n+1}{3}}}^{n-1} I_v = 
8\, I_n\left(\frac{n-1}{2}\right)\left(\frac{n-3}{2}\right) - 16\, I_n \floor{\frac{n-1}{6}}
\]
using the summation identity~\eqref{sum2n+13}.

The combined additional contribution to the sum from Cases 3 and 4 is 
\begin{align*}
16\!\! \sum_{\stack{u,v=1}{\stack{u \ne v}{2v-u\ne n}}}^{n-1} 
\!\!\! I_u\,\, I[u+2v>2n]  						
&= 16 \sum_{u,v=1}^{n-1} I_u\,\, I[u+2v>2n] - 
  16 \!\!\!\! \sum_{u=\ceiling{\frac{2n+1}{3}}}^{n-1} \!\!\!\! I_u -
  16\, I_n \!\!\!\! \sum_{v=\ceiling{\frac{3n+1}{4}}}^{n-1} \!\!\! 1 	\\
&= 16\cdot\frac{1}{2}\floor{\frac{n}{2}} \floor{\frac{n-2}{2}} -16\left(\floor{\frac{n}{6}}+I[n \bmod 6 = 4]\right) - 16 I_n \floor{\frac{n-1}{4}},
\end{align*}
where the first summation is evaluated using the identity \eqref{sumIuIu+2v}, the second using the identity \eqref{sum2n+13}, and the third by considering the cases $n \bmod 4 = 1$ and $n \bmod 4 = 3$ separately.

Add the contribution from Cases 1 and 2 to the additional contribution from Cases 3 and 4 to obtain
\begin{align}
&\sum_{u,v=1}^{n-1} I_u \, \E\big( D_u D_v^2 \big) =		
  8\, I_n\left(\frac{n-1}{2}\right)\left(\frac{n-3}{2}\right) 
+ 8\floor{\frac{n}{2}}\floor{\frac{n-2}{2}} \nonumber \\
&\quad\quad\quad - 16 \floor{\frac{n}{6}} 
- 16 I_n \floor{\frac{n-1}{6}} 
- 16I_n\floor{\frac{n-1}{4}}
- 16 \cdot I[n \bmod 6 = 4].
\label{sum12}
\end{align}

\item[The sum $\sum_{u,v=1}^{n-1} \E\big( D_u^2 D_v^2 \big)$.]
From \eqref{Ddefn}, this sum equals
\[
 16 \sum_{u,v=1}^{n-1} \,\, \sum_{0 \le j,k < \frac{u-1}{2}} \, \sum_{0 \le \ell,m < \frac{v-1}{2}} \E (a_j a_{j+n-u} a_k a_{k+n-u} a_\ell a_{\ell+n-v} a_m a_{m+n-v}).
\]

We distinguish five mutually disjoint cases.
\begin{description}

\item[Case 1:] $j=k$ and $\ell=m$.
The contribution to the sum is
\[
 16 \sum_{u,v=1}^{n-1} \,\, \sum_{0 \le j < \frac{u-1}{2}} \, \sum_{0 \le \ell < \frac{v-1}{2}}1 
= 16 \left (\sum_{u=1}^{n-1} \floor{\frac{u}{2}} \right )^2 
= 16 \floor{\frac{n}{2}}^2 \floor{\frac{n-1}{2}}^2
\]
using the summation identity \eqref{sumflooru}.

\item[Case 2:] $j=k$ and $\ell \ne m$.
By symmetry in $\ell$ and $m$, the contribution to the sum is
\[
 2 \cdot 16 \sum_{u,v=1}^{n-1} \,\, \sum_{0 \le j < \frac{u-1}{2}} \, \sum_{0 \le \ell < m < \frac{v-1}{2}}\E (a_\ell a_{\ell+n-v} a_m a_{m+n-v}),
\]
and the inner sum over $\ell < m$ is zero because by \eqref{two-pairs} the index $\ell$ cannot form an equal or symmetric pair with any of the other three indices $\ell+n-v$, $m$, $m+n-v$ in the given range $\ell < m < \frac{v-1}{2}$.

\item[Case 3:] $j \ne k$ and $\ell=m$.
Similarly to Case 2, the contribution to the sum is zero.

\item[Case 4:] $j \ne k$ and $\ell \ne m$ and $u = v$.
By symmetry in $j,k$ and in $\ell,m$, the contribution to the sum is
\[
 4 \cdot 16 \sum_{u=1}^{n-1} \,\, \sum_{0 \le j < k < \frac{u-1}{2}} \, \sum_{0 \le \ell < m < \frac{u-1}{2}} \E (a_j a_{j+n-u} a_k a_{k+n-u} a_\ell a_{\ell+n-u} a_m a_{m+n-u}).
\]
If $j \ne \ell$, then by \eqref{two-pairs} the expectation term is zero because the smaller of $j, \ell$ cannot form a pair with any of the other seven indices. We may therefore take $j = \ell$, and then by similar reasoning take $k = m$, so that the contribution to the sum is
\[
 64 \sum_{u=1}^{n-1} \,\, \sum_{0 \le j < k < \frac{u-1}{2}} \!\!\! 1 =
 32 \sum_{u=1}^{n-1} \floor{\frac{u}{2}} \floor{\frac{u-2}{2}}
 = \frac{64}{3}\floor{\frac{n}{2}} \left(\frac{n-2}{2}\right) \ceiling{\frac{n-4}{2}}
\]
using the summation identity~\eqref{sumflooruu-2}.

\item[Case 5:] $j \ne k$ and $\ell \ne m$ and $u \ne v$.
By symmetry in $u,v$ and in $j,k$, the contribution to the sum is 
\[
 4 \cdot 16 \sum_{\stack{u,v=1}{u < v}}^{n-1} \,\, \sum_{0 \le j < k < \frac{u-1}{2}} \, \sum_{\stack{0 \le \ell,m < \frac{v-1}{2}}{\ell \ne m}} \E (a_j a_{j+n-u} a_k a_{k+n-u} a_\ell a_{\ell+n-v} a_m a_{m+n-v}).
\]
By \eqref{two-pairs}, the index $j$ cannot form an equal or symmetric pair with any of the indices $j+n-u, k, k+n-u$ in the given range $j < k < \frac{u-1}{2}$. Furthermore, $j$ cannot form a symmetric pair with any of the indices $\ell, \ell+n-v, m, m+n-v$ in the given ranges $\ell, m < \frac{v-1}{2}$ and $j < \frac{u-1}{2}$ and $u < v$. A matching decomposition for the eight indices of the expectation term therefore requires that $j$ form an equal pair with one of the four indices $\ell, \ell+n-v, m, m+n-v$. By symmetry in $\ell, m$, we may replace the resulting four contributions to the sum by twice the contribution from $j = \ell$ and twice the contribution from $j = \ell+n-v$. 

We claim that the contribution from $j = \ell+n-v$ is zero. To prove the claim, set $j = \ell+n-v$, so that $\ell$ and $k$ are now constrained via $\ell+n-v<k<\frac{u-1}{2}$ and the remaining six unpaired indices are $\ell+2n-u-v, k, k+n-u, \ell, m, m+n-v$. It is straightforward to check that $\ell$ cannot form an equal or symmetric pair with any of the three indices $\ell+2n-u-v, k, k+n-u$ subject to the given constraint $\ell+n-v<k<\frac{u-1}{2}$, and therefore by \eqref{two-pairs} the only possible pairing involving $\ell$ is $\ell=m+n-v$. We therefore set $\ell = m+n-v$, so that $m$ and $k$ are now constrained via $m+2n-2v < k < \frac{u-1}{2}$ and the remaining four unpaired indices are $m+3n-u-2v, k, k+n-u, m$. By \eqref{one-pair}, the indices $k, k+n-u$ cannot form an equal or symmetric pair and so, for a matching decomposition, $m$ must form an equal or symmetric pair with $k$ or $k+n-u$. This is not possible subject to the given constraint $m+2n-2v < k < \frac{u-1}{2}$, proving the claim.

The contribution to the sum from Case 5 is therefore twice the contribution from $j = \ell$, namely
\[
 128 \sum_{\stack{u,v=1}{u < v}}^{n-1} \,\, \sum_{0 \le j < k < \frac{u-1}{2}} \, \sum_{\stack{0 \le m < \frac{v-1}{2}}{m \ne j}} \E (a_k a_{k+n-u} a_{j+n-u} a_{j+n-v} a_m a_{m+n-v}).
\]
The index $k$ cannot form a symmetric pair with any of the indices $k+n-u,j+n-u,j+n-v,m,m+n-v$ in the given ranges $j < k < \frac{u-1}{2}$ and $m < \frac{v-1}{2}$ and $u<v$. A matching decomposition for the six indices of the expectation term therefore requires that $k$ form an equal pair with one of the four indices $j+n-u,j+n-v,m,m+n-v$. We now determine all matching decompositions by starting with each of these four equal pairs in turn, identifying all possible arrangements of the four remaining unpaired indices into two further pairs consistent with the given ranges for $u,v,j,k,m$. These arrangements are listed in Cases 5a to 5e below. Inconsistent arrangements of indices arising from combinations of two or more pairings are summarized in Table~\ref{tab:inconsistent22}.

\begin{description}
\item[Case 5a:] $k = j+n-u$ and $m = j+n-v$ and $k+n-u \sim m+n-v$.
This gives $j = \frac{2u+2v-3n-1}{2}$ and $k=\frac{2v-n-1}{2}$ and $m=\frac{2u-n-1}{2}$, and the contribution to the sum is
\[
128\, I_n \sum_{\stack{u,v=1}{u < v}}^{n-1} I[2u+2v > 3n]\,\, I[2v-u<n].
\]

\item[Case 5b:] $k = m+n-v$ and $k+n-u \sim j+n-u$ and $j+n-v=m$. 
This gives $j = \frac{2u+2v-3n-1}{2}$ and $k=\frac{2u-2v+n-1}{2}$ and $m=\frac{2u-n-1}{2}$, and the contribution to the sum is
\[
128\, I_n \sum_{\stack{u,v=1}{u<v}}^{n-1} I[2u+2v > 3n]\,\, I[2v-u>n].
\]

\item[Case 5c:] $k = j+n-v$ and $m = j+n-u$ and $k+n-u,m+n-v$ form an equal or symmetric pair.
This gives $k = j+n-v$ and $m=j+n-u$ when the third pair (formed by $k+n-u,m+n-v$) is equal; the special case $j=\frac{2u+2v-3n-1}{2}$, $k=\frac{2u-n-1}{2}$, $m=\frac{2v-n-1}{2}$ is obtained when this pair is symmetric, and so is not counted again. The contribution to the sum is
\[
128 \sum_{\stack{u,v=1}{u < v}}^{n-1} \,\, \sum_{j \ge 0} I[j < \tfrac{2u+v-2n-1}{2}].
\]

\item[Case 5d:] $k = j+n-v$ and $m \sim k+n-u$ and $j+n-u, m+n-v$ form an equal or symmetric pair.
This gives $k=j+n-v$ and $m=u+v-n-1-j$ when the third pair is symmetric; the special case $j=\frac{2u-n-1}{2}$, $k=\frac{2u-2v+n-1}{2}$, $m=\frac{2v-n-1}{2}$ is obtained when this pair is equal, and is not counted again. The contribution to the sum is
\[
128 \sum_{\stack{u,v=1}{u < v}}^{n-1} \,\, \sum_{j \ge 0} I[\tfrac{2u+v-2n-1}{2} < j < \tfrac{u+2v-2n-1}{2}].
\]

\item[Case 5e:] $k = m$ and $k+n-u \sim j+n-v$ and $j+n-u, m+n-v$ form an equal or symmetric pair.
This gives $k = m = u+v-n-1-j$ when the third pair is symmetric; the special case $j=\frac{2u-n-1}{2}$, $k=m=\frac{2v-n-1}{2}$ is obtained when this pair is equal, and is not counted again. The contribution to the sum is
\[
128 \sum_{\stack{u,v=1}{u < v}}^{n-1} \,\, \sum_{j \ge 0} I[\tfrac{u+2v-2n-1}{2} < j < \tfrac{u+v-n-1}{2}].
\]
\end{description}

\begin{table}[htb]
\centering
\caption{Inconsistent index pairings for Case 5 of the sum $\sum_{u,v=1}^{n-1} \E\big( D_u^2 D_v^2 \big)$}
\label{tab:inconsistent22}
\vspace{1ex}
\begin{tabular}{|lll|l|}						  \hline
Index pairings 		& 			&			
	& Inconsistency from 					       	\\\hline
			&			&		&	\\[-2ex]
$k=j+n-u$, 		& $m=k+n-u$, 		& $j+n-v \sim m+n-v$ 	
	& $m=\frac{v-1}{2}+\frac{v-u}{2}+\frac{n-u}{2} > \frac{v-1}{2}$	\\[1ex]
$k=j+n-u$, 		& $m \sim k+n-u$,	& $j+n-v \sim m+n-v$ 	
	& $u=v$ 							\\[1ex]
$k=j+n-u$, 		& $m = j+n-v$,		& $k+n-u = m+n-v$ 	
	& $u=v$ 							\\[1ex]
$k=j+n-u$, 		& $m \sim j+n-v$,	& $k+n-u = m+n-v$ 	
	& $k=\frac{n-1}{2}$ 						\\[1ex]
$k=j+n-u$, 		& $m \sim j+n-v$,	& $k+n-u \sim m+n-v$ 	
	& $u=n$ 							\\[1ex] 
\hline 			&			&		&	\\[-2ex]
$k=m+n-v$, 		& $k+n-u \sim j+n-u$, 	& $j+n-v \sim m$ 	
	& $u=n$								\\[1ex]
$k=m+n-v$, 		& $k+n-u \sim j+n-v$, 	& $j+n-u = m$ 	
	& $k = \frac{n-1}{2}$						\\[1ex]
$k=m+n-v$, 		& $k+n-u \sim j+n-v$, 	& $j+n-u \sim m$ 	
	& $v=n$								\\[1ex]
$k=m+n-v$, 		& $k+n-u = m$ 		& 
	& $u+v=2n$							\\[1ex]
$k=m+n-v$, 		& $k+n-u \sim m$ 	&
	& $k = \frac{u-1}{2}+\frac{n-v}{2} > \frac{u-1}{2}$		\\[1ex]
\hline 			&			&		&	\\[-2ex]
$k=j+n-v$, 		& $m \sim j+n-u$, 	& $k+n-u = m+n-v$ 	
	& $m = \frac{n-1}{2}$						\\[1ex]
$k=j+n-v$, 		& $m \sim j+n-u$,	& $k+n-u \sim m+n-v$ 	
	& $v=n$ 							\\[1ex]
$k=j+n-v$, 		& $m = k+n-u$,		& $j+n-u = m+n-v$ 	
	& $v=n$ 							\\[1ex]
$k=j+n-v$, 		& $m = k+n-u$,		& $j+n-u \sim m+n-v$ 	
	& $m=\frac{n-1}{2}$ 						\\[1ex]
\hline 			&			&		&	\\[-2ex]
$k=m$,	 		& $k+n-u \sim j+n-u$,	& $j+n-v \sim m+n-v$ 	
	& $u=v$						 		\\[1ex]
$k=m$,	 		& $k+n-u = m+n-v$ 	& 
	& $u=v$					 			\\[1ex]
$k=m$,	 		& $k+n-u \sim m+n-v$, 	& $j+n-u \sim j+n-v$
	& $j=k$						 		\\[1ex]
\hline
\end{tabular}
\end{table}
\end{description}

We now calculate the total contribution to the sum from Cases 1 to 5.
The combined contribution from Cases 5a and 5b is 
\[
 128\, I_n \sum_{\stack{u,v=1}{u<v}}^{n-1} I[2u+2v > 3n]\,\, I[2v-u \ne n] 
 = 128\, I_n \sum_{\stack{u,v=1}{u<v}}^{n-1} I[2u+2v > 3n]
 - 128\, I_n \!\! \sum_{v=\ceiling{\frac{5n+1}{6}}}^{n-1} 1.
\]
Since each summand of the sum over $u,v$ is symmetric in $u$ and $v$, this combined contribution is
\begin{align*}
 & 128 \cdot \frac{1}{2} I_n \sum_{u,v=1}^{n-1} I[2u+2v > 3n] - 64\, I_n \!\! \sum_{u=\ceiling{\frac{3n+1}{4}}}^{n-1} \!\! 1 - 128\, I_n \!\!\sum_{v=\ceiling{\frac{5n+1}{6}}}^{n-1} \!\!1 \\
 &\quad\quad= 32\, I_n\left(\frac{n-1}{2}\right)\left(\frac{n-3}{2}\right) -64\, I_n \floor{\frac{n-1}{4}} -128\, I_n \floor{\frac{n-1}{6}}.
\end{align*}

The combined contribution to the sum from Cases 5c, 5d, and 5e is 
\begin{align*}
& 128 \sum_{\stack{u,v=1}{u<v}}^{n-1} \sum_{j \ge 0}
I[j < \tfrac{u+v-n-1}{2},\, j \ne \tfrac{2u+v-2n-1}{2},\, j \ne \tfrac{u+2v-2n-1}{2}] \\
&\quad\quad = 128 \sum_{\stack{u,v=1}{u<v}}^{n-1} \left(\floor{\frac{u+v-n}{2}}I[u+v>n] -I_v\, I[2u+v > 2n] - I_u\, I[u+2v > 2n] \right),
\end{align*}
which is of the form $\displaystyle{128 \sum_{\stack{u,v=1}{u<v}}^{n-1} s_{u,v}}$ where $s_{u,v}$ is symmetric in $u$ and $v$. We calculate this combined contribution as $64\sum_{u,v=1}^{n-1} s_{u,v} - 64\sum_{u=1}^{n-1} s_{u,u}$. We have
\begin{align*}
64 \sum_{u,v=1}^{n-1}s_{u,v} 
  &= 64 \sum_{u,v=1}^{n-1} \floor{\frac{u+v-n}{2}}I[u+v>n] -128 \sum_{u,v=1}^{n-1} I_u\, I[u+2v > 2n] \\
  &= 64\left( \frac{2}{3}\floor{\frac{n}{2}}\left(\frac{n-2}{2}\right)\ceiling{\frac{n-4}{2}} + \frac{1}{2}\floor{\frac{n}{2}}\floor{\frac{n-2}{2}} \right) - 128 \cdot \frac{1}{2}\floor{\frac{n}{2}}\floor{\frac{n-2}{2}},
\end{align*}
where the second summation is evaluated using the identity \eqref{sumIuIu+2v}, and the first by substituting $w=u+v-n$ to obtain
$64\sum_{u=1}^{n-1}\sum_{w=1}^{u-1}\lfloor\frac{w}{2}\rfloor = 
 64\sum_{u=1}^{n-1}\lfloor\frac{u}{2}\rfloor \lfloor\frac{u-1}{2}\rfloor$ from the identity~\eqref{sumflooru}, then applying the identity
\[
 \floor{\frac{u}{2}}\floor{\frac{u-1}{2}} = \floor{\frac{u}{2}}\floor{\frac{u-2}{2}} + I_u \left(\frac{u-1}{2}\right),
\]
and finally using the identities \eqref{sumflooruu-2} and~\eqref{sumIuu-1}. We also have
\begin{align*}
64\sum_{u=1}^{n-1} s_{u,u} 
  &=64 \!\! \sum_{u=\ceiling{\frac{n+1}{2}}}^{n-1} \floor{\frac{2u-n}{2}} - 128 \!\! \sum_{u=\ceiling{\frac{2n+1}{3}}}^{n-1} I_u \\
  &= 32 \floor{\frac{n}{2}}\floor{\frac{n-2}{2}} -128 \left ( \floor{\frac{n}{6}} + I[n \bmod 6 = 4] \right ),
\end{align*}
where the first summation is evaluated by considering the cases $n$ even and $n$ odd separately, and the second using the identity \eqref{sum2n+13}.

The combined contribution to the sum from Cases 5c, 5d, and 5e is therefore
\begin{align*}
& 64\sum_{u,v=1}^{n-1}s_{u,v} - 64\sum_{u=1}^{n-1}s_{u,u} \\
& \quad\quad = \frac{128}{3}\floor{\frac{n}{2}}\left(\frac{n-2}{2}\right)\ceiling{\frac{n-4}{2}} - 64\floor{\frac{n}{2}}\floor{\frac{n-2}{2}}
+128 \floor{\frac{n}{6}} + 128 \cdot I[n \bmod 6 = 4].
\end{align*}

Add the contributions from Case 1, Case 4, Cases 5a/5b, and Cases 5c/5d/5e to obtain
\begin{align}
&\sum_{u,v=1}^{n-1} \E\big( D_u^2 D_v^2 \big) =
  16\floor{\frac{n}{2}}^2\floor{\frac{n-1}{2}}^2
 +64\floor{\frac{n}{2}}\left(\frac{n-2}{2}\right)\ceiling{\frac{n-4}{2}}
 -64\floor{\frac{n}{2}}\floor{\frac{n-2}{2}} 			
\nonumber \\
& +32\,I_n\left(\frac{n-1}{2}\right)\left(\frac{n-3}{2}\right)
 -64\,I_n\floor{\frac{n-1}{4}}
-128\,I_n\floor{\frac{n-1}{6}}
+128\floor{\frac{n}{6}} +128 \cdot I[n \bmod 6 = 4].
\label{sum22}
\end{align}

\end{description}

We are now ready to determine $n^2\, \E (\frac{1}{F(A)})$ and $n^4\, \Var (\frac{1}{F(A)})$ using Proposition~\ref{prop:expressions}, separating the calculation according to whether $n$ is even or odd. Substitution of \eqref{sum1} and \eqref{sum2} into \eqref{E-symm} gives 
\begin{equation}
E = \frac{1}{2}(2n^2-3n+I_n),
\label{E-symm-res}
\end{equation}
and then from \eqref{Ef44} we obtain 
\[
n^2\, \E \Big(\frac{1}{F(A)}\Big) = 2n^2-3n+\frac{1-(-1)^n}{2},
\] 
as required.
Substitution of \eqref{sum1}, \eqref{sum11}, \eqref{sum2}, \eqref{sum12}, and \eqref{sum22} into \eqref{V-symm} gives, after simplification,
\[
V = \begin{cases}
\displaystyle{n^4+5n^3-\frac{207}{4}n^2+76n+64\floor{\frac{n}{6}}+64\cdot I[n \bmod 6 =4]}	& \mbox{for $n$ even,} \\[2ex]
\displaystyle{n^4+5n^3-\frac{131}{4}n^2+\frac{77}{2}n-128\floor{\frac{n-1}{6}}-144\floor{\frac{n-1}{4}}-\frac{47}{4}}
	& \mbox{for $n$ odd},
\end{cases}
\]
and then from \eqref{Vf44} and \eqref{E-symm-res} we find that
\[
n^4\, \Var \Big(\frac{1}{F(A)}\Big) = 
\begin{cases}
   \displaystyle{32n^3 - 216n^2 + 304n 
   +256\floor{\frac{n}{6}}
   +256 \cdot I[n \bmod 6 = 4]} 	& \mbox{for $n$ even,} \\[2ex]
   \displaystyle{32n^3 - 144n^2 + 160n 
   -576\floor{\frac{n-1}{4}}
   -512\floor{\frac{n-1}{6}}}
   -48					& \mbox{for $n$ odd},
\end{cases}
\]
as required.
\end{proof}

\section{The class $\SSP_{n}$}
\label{sec:skew-symmetric}
In this section, we use Proposition~\ref{prop:expressions} to prove Theorem~\ref{thm:skew-symm-mean-var} for the class $\SSP_n$ of skew-symmetric binary sequences.

\begin{proof}[Proof of Theorem~$\ref{thm:skew-symm-mean-var}$]
We modify the proof of Theorem~\ref{thm:symm-mean-var} to obtain the result.
Let $(a_0,a_1,\dots,a_{n-1}) \in \SSP_n$. By the definition~\eqref{SSPdefn} of $\SSP_n$, the $a_j$ satisfy the skew-symmetry condition
\begin{equation}
a_j = (-1)^{j+\frac{n-1}{2}} a_{n-1-j} \quad \mbox{for $0 \le j <n$}.
\label{skew-symm-condition}
\end{equation}
We regard the sequence entries $a_0, a_1, \dots, a_{\frac{n-1}{2}}$ as independent random variables that each take the values $1$ and $-1$ with probability $\frac{1}{2}$, and the remaining sequence entries as being determined by~\eqref{skew-symm-condition}. 
Use condition \eqref{skew-symm-condition} to rewrite \eqref{Cdefn} as
\[
C_u = I_u \big( (-1)^\frac{n-u}{2} + D_u \big ).
\]
Substitute for $C_u$ in \eqref{E}, and for $C_u$ and $C_v$ in \eqref{V}, to obtain the expressions
\begin{align}
E &= \frac{n-1}{2} + 2\sum_{u=1}^{n-1} (-1)^\frac{n-u}{2} I_u\, \E D_u + \sum_{u=1}^{n-1} I_u \, \E  D_u^2, 						\label{E-skew-symm} \\
V &= \left(\frac{n-1}{2}\right)^2
   +4\left(\frac{n-1}{2}\right) \sum_{u=1}^{n-1} (-1)^\frac{n-u}{2} I_u \, \E D_u
   +4 \sum_{u,v=1}^{n-1} (-1)^\frac{2n-u-v}{2} I_u I_v \, \E\big( D_uD_v \big) 									\nonumber \\
  &\phantom{=} +2\left(\frac{n-1}{2}\right) \sum_{u=1}^{n-1} I_u \, \E D_u^2 	
  +4 \sum_{u,v=1}^{n-1} (-1)^\frac{n-u}{2} I_u I_v\, \E\big( D_u D_v^2 \big) 
  +\sum_{u,v=1}^{n-1} I_u I_v \, \E\big( D_u^2 D_v^2 \big), \label{V-skew-symm}
\end{align}
noting that $(-1)^{n-u} I_u = I_u$ because $n$ is odd here.

We express each of 
$\E D_u$,
$\E\big( D_u D_v \big )$,
$\E D_u^2$,
$\E\big( D_u D_v^2 \big )$, and
$\E\big( D_u^2 D_v^2 \big )$
as a sum of expectation terms of the form $\E(a_{j_1} a_{j_2} \dots a_{j_{2r}})$, where $1 \le r \le 4$, and then calculate $n^2\, \E (\frac{1}{F(A)})$ and $n^4\, \Var (\frac{1}{F(A)})$ from Proposition~\ref{prop:expressions} by substitution into the forms \eqref{E-skew-symm} and~\eqref{V-skew-symm}.
In view of the skew-symmetry condition \eqref{skew-symm-condition}, the expectation term $\E(a_{j_1} a_{j_2} \dots a_{j_{2r}})$ is nonzero exactly when the indices $j_1, j_2, \dots, j_{2r}$ admit a matching decomposition. The index sets admitting a matching decomposition are identical to those in the proof of Theorem~\ref{thm:symm-mean-var}; each symmetric index pair $\{j, k\}$ in the resulting expectation term introduces an additional multiplicative factor $(-1)^{j+\frac{n-1}{2}}$, by~\eqref{skew-symm-condition}.

We use the same case analyses in the following calculations as in the proof of Theorem~\ref{thm:symm-mean-var}, inserting additional factors $I_u$, $I_v$, $(-1)^\frac{n-u}{2}$, and $(-1)^\frac{2n-u-v}{2}$ as appropriate. 

\begin{description}

\item[The sum $\sum_{u=1}^{n-1} (-1)^\frac{n-u}{2} I_u \, \E D_u$.]
Similarly to~\eqref{sum1}, this sum is zero.

\item[The sum $\sum_{u,v=1}^{n-1} (-1)^\frac{2n-u-v}{2} I_u I_v \, \E\big( D_uD_v \big)$.]
We modify the calculation of $\sum_{u,v=1}^{n-1} I_u I_v \, \E\big( D_uD_v \big)$ in the proof of Theorem~\ref{thm:symm-mean-var}. Cases 1 and 2 both receive a multiplicative factor $(-1)^\frac{2n-u-v}{2}I_uI_v$ in place of~$I_uI_v$. Case 2 receives a further multiplicative factor $(-1)^{j+n-u+\frac{n-1}{2}} = (-1)^\frac{v-u}{2}$ because of the symmetric index pair $\{j+n-u,\,k+n-v\}$ with $j=\frac{u+v-n-1}{2}$.
The resulting multiplicative factor in both cases is $I_uI_v$ (using the relation $u=v$ for Case~1), and the sum is therefore unchanged from~\eqref{sum11}. Since $n$ is odd, expression \eqref{sum11} simplifies to give
\[
\sum_{u,v=1}^{n-1} (-1)^\frac{2n-u-v}{2} I_u I_v \, \E\big( D_uD_v \big)
= 4 \left( \frac{n-1}{2} \right) \left( \frac{n-3}{2} \right) - 4 \floor{\frac{n-1}{4}}.
\]

\item[The sum $\sum_{u=1}^{n-1} I_u \, \E D_u^2$.]
We modify the calculation of $\sum_{u=1}^{n-1} \E D_u^2$ in the proof of Theorem~\ref{thm:symm-mean-var} by introducing an additional factor $I_u$, giving
\[
\sum_{u=1}^{n-1} I_u \, \E D_u^2 
 = 4 \sum_{u=1}^{n-1} I_u\, \left ( \frac{u-1}{2} \right )
 = 2 \left(\frac{n-1}{2}\right) \left(\frac{n-3}{2}\right)
\]
using the summation identity~\eqref{sumIuu-1}.

\item[The sum $\sum_{u,v=1}^{n-1} (-1)^\frac{n-u}{2} I_u I_v \, \E\big( D_u D_v^2 \big)$.] 
We modify the calculation of $\sum_{u,v=1}^{n-1} I_u \, \E\big( D_uD_v^2 \big)$ in the proof of Theorem~\ref{thm:symm-mean-var}. Cases 1 to 4 all receive a multiplicative factor $(-1)^\frac{n-u}{2} I_u I_v$ in place of~$I_u$. The presence of symmetric index pairs introduces further multiplicative factors:
in Case 1, a factor of $(-1)^{(k+n-v)+\frac{n-1}{2}}(-1)^{(j+n-u)+\frac{n-1}{2}} = (-1)^\frac{n-u}{2}$ because of index pairs $\{k+n-v,\,\ell+n-v\}$ and $\{j+n-u,\,\ell\}$ with $j=k$ and $2v-u=n$;
in Case 3, a factor of $(-1)^{(j+n-u)+\frac{n-1}{2}} = (-1)^\frac{2v-u+n-2}{2}$ because of index pair $\{j+n-u,\,\ell+n-v\}$ with $j=\frac{u+2v-2n-1}{2}$;
and in Case 4, a factor of $(-1)^{(k+n-v)+\frac{n-1}{2}} = (-1)^\frac{u+n-2}{2}$ because of index pair $\{k+n-v,\,\ell+n-v\}$ with $k=\frac{u+2v-2n-1}{2}$.

The resulting multiplicative factors are given in Table~\ref{tab:factors12} (using the relation $2v-u=n$ to evaluate Case~2). 
We see that the calculation differs from that of $\sum_{u,v=1}^{n-1} I_u \, \E\big( D_uD_v^2 \big)$ in the proof of Theorem~\ref{thm:symm-mean-var} only via the introduction of a factor $I_v$ in all four cases.

\begin{table}[htb]
\centering
\caption{Multiplicative factors in calculation of
$\sum_{u,v=1}^{n-1} (-1)^\frac{n-u}{2} I_u I_v \, \E\big( D_uD_v^2 \big)$}
\label{tab:factors12}
\vspace{1ex}
\begin{tabular}{|l|l|l|}					          \hline
Case	& Multiplicative factor		& Evaluates to 			\\\hline
	&	 			&				\\[-2ex]
1	& $(-1)^\frac{n-u}{2}I_uI_v \cdot (-1)^\frac{n-u}{2}$		
					& $I_uI_v$			\\[1ex]
2	& $(-1)^\frac{n-u}{2}I_uI_v$	& $I_uI_v$			\\[1ex] 
3	& $(-1)^\frac{n-u}{2}I_uI_v \cdot (-1)^\frac{2v-u+n-2}{2}$		
		 			& $I_uI_v$			\\[1ex]
4	& $(-1)^\frac{n-u}{2}I_uI_v \cdot (-1)^\frac{u+n-2}{2}$		
					& $I_uI_v$			\\[1ex]
\hline
\end{tabular}
\end{table}

The combined contribution from Cases 1 and 2 is therefore
\[
 16 \sum_{v=\frac{n+1}{2}}^{n-1} I_v \left(\frac{2v-n-1}{2} \right)
-16 \!\! \sum_{v=\ceiling{\frac{2n+1}{3}}}^{n-1} I_v = 
16\floor{\frac{n-1}{4}}\floor{\frac{n-3}{4}} - 16 \floor{\frac{n-1}{6}}
\]
using the summation identities \eqref{sumIu2u-n-1} and~\eqref{sum2n+13}.
The combined additional contribution from Cases 3 and 4 is
\begin{align*}
16 \sum_{\stack{u,v=1}{\stack{u \ne v}{2v-u\ne n}}}^{n-1} 
\!\!\! I_uI_v\,\, I[u+2v>2n]
&= 16 \sum_{u,v=1}^{n-1} I_uI_v\,\, I[u+2v>2n] - 
  16 \!\!\! \sum_{u=\ceiling{\frac{2n+1}{3}}}^{n-1} \!\!\!\! I_u -
  16 \!\!\!\! \sum_{v=\ceiling{\frac{3n+1}{4}}}^{n-1} \!\!\! I_v 	\\
&= 16\floor{\frac{n-1}{4}}\floor{\frac{n-3}{4}} -16\floor{\frac{n-1}{6}} - 16 \floor{\frac{n-1}{8}}
\end{align*}
using the summation identities \eqref{sumIuIvI2u+v}, \eqref{sum2n+13}, and~\eqref{sum3n+14}.

Add the contributions from Cases 1 and 2 to the additional contribution from Cases 3 and 4 to obtain
\[
\sum_{u,v=1}^{n-1} (-1)^\frac{n-u}{2} I_uI_v \, \E\big( D_u D_v^2 \big) =
32\floor{\frac{n-1}{4}}\floor{\frac{n-3}{4}} - 32 \floor{\frac{n-1}{6}} - 16 \floor{\frac{n-1}{8}}.
\]

\item[The sum $\sum_{u,v=1}^{n-1} I_u I_v \, \E\big( D_u^2 D_v^2 \big)$.]
We modify the calculation of $\sum_{u,v=1}^{n-1} \E\big( D_u^2D_v^2 \big)$ in the proof of Theorem~\ref{thm:symm-mean-var}. All cases receive a multiplicative factor $I_u I_v$ in place of~$1$. The presence of symmetric index pairs introduces further multiplicative factors:
in Case 5a, a factor of $(-1)^{(k+n-u)+\frac{n-1}{2}} = 1$ because of index pair $\{k+n-u,\,m+n-v\}$ with $k=\frac{2v-n-1}{2}$;
in Case 5b, a factor of $(-1)^{(k+n-u)+\frac{n-1}{2}} = 1$ because of index pair $\{k+n-u,\,j+n-u\}$ with $k=\frac{2u-2v+n-1}{2}$;
in Case 5d, a factor of $(-1)^m (-1)^{j+n-u} = 1$ because of index pairs $\{m,\,k+n-u\}$ and $\{j+n-u,\,m+n-v\}$ with $m=u+v-n-1-j$;
and in Case 5e, a factor of $(-1)^{k+n-u} (-1)^{j+n-u} = 1$ because of index pairs $\{k+n-u,\,j+n-v\}$ and $\{j+n-u,\,m+n-v\}$ with $k=u+v-n-1-j$.
Since these further factors all equal~1, we see that the calculation differs from that of $\sum_{u,v=1}^{n-1} \E\big( D_uD_v^2 \big)$ in the proof of Theorem~\ref{thm:symm-mean-var} only via the introduction of a factor $I_uI_v$ in all cases.

The contribution from Case 1 is therefore
\[
16 \left (\sum_{u=1}^{n-1} I_u \left(\frac{u-1}{2}\right) \right )^2 = 
4\left(\frac{n-1}{2}\right)^2\left(\frac{n-3}{2}\right)^2
\]
using the summation identity~\eqref{sumIuu-1}.
The contribution from Cases 2 and~3 is zero.
The contribution from Case 4 is
\[
 32 \sum_{u=1}^{n-1} I_u \left(\frac{u-1}{2}\right)\left(\frac{u-3}{2}\right)
 = \frac{32}{3}\left(\frac{n-1}{2}\right)\left(\frac{n-3}{2}\right)\left(\frac{n-5}{2}\right)
\]
using the summation identity~\eqref{sumIuu-1u-3}.
The contribution from Cases 5a and 5b is
\begin{align*}
& 128 \sum_{\stack{u,v=1}{u<v}}^{n-1} I_uI_v \, I[2u+2v > 3n]\,\, I[2v-u \ne n] \\
  & \quad\quad= 128\cdot \frac{1}{2} \sum_{u,v=1}^{n-1} I_uI_v \, I[2u+2v > 3n] - 64 \!\!\!\! \sum_{u=\ceiling{\frac{3n+1}{4}}}^{n-1} \!\! I_u - 128 \!\!\!\!\sum_{v=\ceiling{\frac{5n+1}{6}}}^{n-1} \!\! I_v \\
  & \quad\quad= 32\floor{\frac{n-1}{4}}\floor{\frac{n-5}{4}} - 64 \floor{\frac{n-1}{8}} -128 \floor{\frac{n-1}{12}},
\end{align*}
where the first summation is evaluated by substituting $u=2U+1$ and $v=2V+1$ and considering the cases $n \bmod 4 = 1$ and $n \bmod 4 = 3$ separately, the second using identity~\eqref{sum3n+14}, and the third using an identity analogous to~\eqref{sum3n+14}.
The combined contribution to the sum from Cases 5c, 5d, and 5e is 
\[
128 \sum_{\stack{u,v=1}{u<v}}^{n-1} \left(I_uI_v\left(\frac{u+v-n-1}{2}\right)I[u+v>n] -I_uI_v\, I[2u+v > 2n] - I_uI_v\, I[u+2v > 2n] \right),
\]
which is of the form $\displaystyle{128 \sum_{\stack{u,v=1}{u<v}}^{n-1} s'_{u,v}}$ where $s'_{u,v}$ is symmetric in $u$ and $v$. We calculate 
\begin{align*}
64 \sum_{u,v=1}^{n-1}s'_{u,v} 
  &= 64 \sum_{u,v=1}^{n-1} I_uI_v\left(\frac{u+v-n-1}{2}\right)I[u+v>n] -128 \sum_{u,v=1}^{n-1} I_uI_v\, I[u+2v > 2n] \\
  &= 64 \cdot \frac{1}{6}\left(\frac{n-1}{2}\right)\left(\frac{n-3}{2}\right)\left(\frac{n-5}{2}\right) - 128 \floor{\frac{n-1}{4}}\floor{\frac{n-3}{4}},
\end{align*}
where the first summation is evaluated by substituting $w=u+v-n$ to obtain the expression $64\sum_{u=1}^{n-1}I_u\sum_{w=1}^{u-1}I_w \left(\frac{w-1}{2}\right)$ and then using the identities~\eqref{sumIuu-1} and~\eqref{sumIuu-1u-3}, and the second using the identity~\eqref{sumIuIvI2u+v}; and
\begin{align*}
64\sum_{u=1}^{n-1} s'_{u,u} 
  &=64 \!\! \sum_{u=\frac{n+1}{2}}^{n-1} I_u\left(\frac{2u-n-1}{2}\right) - 128 \!\! \sum_{u=\ceiling{\frac{2n+1}{3}}}^{n-1} I_u \\
  &= 64 \floor{\frac{n-1}{4}}\floor{\frac{n-3}{4}} -128 \floor{\frac{n-1}{6}},
\end{align*}
using the summation identities~\eqref{sumIu2u-n-1} and~\eqref{sum2n+13}.
The combined contribution to the sum from Cases 5c, 5d, and 5e is 
\begin{align*}
& 64\sum_{u,v=1}^{n-1}s'_{u,v} - 64\sum_{u=1}^{n-1}s'_{u,u} \\
& \quad\quad =
  \frac{32}{3}\left(\frac{n-1}{2}\right)\left(\frac{n-3}{2}\right)\left(\frac{n-5}{2}\right) - 192 \floor{\frac{n-1}{4}}\floor{\frac{n-3}{4}} +128 \floor{\frac{n-1}{6}}.
\end{align*}
Add the contributions from Case 1, Case 4, Cases 5a/5b, and Cases 5c/5d/5e to obtain
\begin{align*}
& \sum_{u,v=1}^{n-1} I_u I_v \, \E\big( D_u^2 D_v^2 \big) \\
&\quad\quad = 4\left(\frac{n-1}{2}\right)^2\left(\frac{n-3}{2}\right)^2
  + \frac{64}{3}\left(\frac{n-1}{2}\right)\left(\frac{n-3}{2}\right)\left(\frac{n-5}{2}\right) 
 - 192 \floor{\frac{n-1}{4}}\floor{\frac{n-3}{4}} \\
&\quad\quad \phantom{=}+ 32\floor{\frac{n-1}{4}}\floor{\frac{n-5}{4}} 
 +128 \floor{\frac{n-1}{6}} - 64 \floor{\frac{n-1}{8}} -128 \floor{\frac{n-1}{12}}.
\end{align*}

\end{description}

We now use the calculated expectation expressions to determine $n^2\, \E (\frac{1}{F(A)})$ and $n^4\, \Var (\frac{1}{F(A)})$ using Proposition~\ref{prop:expressions}. Substitution into \eqref{E-skew-symm} gives 
\begin{equation}
E = \frac{1}{2}(n^2-3n+2),
\label{E-skew-symm-res}
\end{equation}
and then from \eqref{Ef44} we obtain 
\[
n^2\, \E \Big( \frac{1}{F(A)} \Big) = n^2-3n+2,
\] 
as required. Substitution into \eqref{V-skew-symm}, after separation according to whether $n \bmod 4 = 1$ or $n \bmod 4 = 3$, gives
\[
V = \begin{cases}
\displaystyle{\frac{1}{4}n^4+\frac{7}{6}n^3-\frac{75}{4}n^2+\frac{151}{3}n-128\floor{\frac{n-1}{8}}-128\floor{\frac{n-1}{12}}}-33 	& \mbox{for $n \bmod 4 = 1$,} \\[2ex]
\displaystyle{\frac{1}{4}n^4+\frac{7}{6}n^3-\frac{75}{4}n^2+\frac{127}{3}n-128\floor{\frac{n-1}{8}}-128\floor{\frac{n-1}{12}}}-9 	& \mbox{for $n \bmod 4 = 3$,} 
\end{cases}
\]
and then from \eqref{Vf44} and \eqref{E-skew-symm-res} we find that
\[
n^4\, \Var \Big( \frac{1}{F(A)} \Big) = 
\frac{32}{3}n^3 - 88n^2 + \frac{592}{3}n 
-512\floor{\frac{n-1}{8}}
-512\floor{\frac{n-1}{12}}
-88+16(-1)^\frac{n-1}{2}(n-3),
\]
as required.

\end{proof}

\section*{Acknowledgements}
I am most grateful to Kai-Uwe Schmidt for many helpful discussions, which significantly improved this paper. I was deeply saddened both personally and professionally by his untimely passing in August 2023 and I respectfully dedicate this paper to his memory.


\end{document}